\documentclass[11pt]{article}
\usepackage{amssymb}
\usepackage{graphicx}
\usepackage{xcolor} 
\usepackage{tensor}
\usepackage{fullpage} 
\usepackage{amsmath}
\usepackage{amsthm}
\usepackage{verbatim}
\usepackage{hyperref}
\usepackage{enumitem}
\setlist[enumerate]{leftmargin=1.2em}
\setlist[itemize]{leftmargin=1.2em}

\usepackage{seqsplit,mathtools}

\definecolor{green}{rgb}{0,0.8,0} 

\newtheorem{theorem}{Theorem}[section]
\newtheorem{corollary}[theorem]{Corollary}
\newtheorem{lemma}[theorem]{Lemma}
\newtheorem{proposition}[theorem]{Proposition}

\theoremstyle{definition}
\newtheorem{definition}[theorem]{Definition}

\theoremstyle{remark}
\newtheorem{remark}[theorem]{Remark}

\numberwithin{equation}{section}

\newcommand{\nnrm}[1]{{\vert\kern-0.25ex\vert\kern-0.25ex\vert #1 
		\vert\kern-0.25ex\vert\kern-0.25ex\vert}}

\newcommand{\supp}{{\mathrm{supp}}\,}

\newcommand{\varep}{\varepsilon}

\newcommand{\omg}{\omega}

\newcommand{\ep}{\varepsilon}

\newcommand{\la}{\lambda}
\newcommand{\R}{\mathbb{R}}

\begin{document}
	\bibliographystyle{plain}
	\title{Existence and Stability of  the Lamb Dipoles  for the
		Quasi-Geostrophic Shallow-Water Equations} 
	\author{ Shanfa Lai\thanks{Institute of Applied Mathematics, Chinese Academy of Sciences, Beijing
			100190, and University of Chinese Academy of Sciences, Beijing 100049,
			P.R. China. Email: laishanfa@amss.ac.cn. }
		\and Guolin Qin\thanks{Institute of Applied Mathematics, Chinese Academy of Sciences, Beijing
			100190, and University of Chinese Academy of Sciences, Beijing 100049,
			P.R. China. Email: qinguolin18@mails.ucas.ac.cn.}
		\and Weicheng Zhan \thanks{ School of Mathematical Sciences, Xiamen University, Xiamen 361005, Fujian, People’s Republic of
			China. Email: zhanweicheng@amss.ac.cn.}
	}	
	
	\date\today
	
	\maketitle

	\renewcommand{\thefootnote}{\fnsymbol{footnote}}
	\renewcommand{\thefootnote}{\arabic{footnote}}

	\begin{abstract} 
		In this paper, we prove the nonlinear orbital stability of vortex dipoles for the quasi-geostrophic shallow-water (QGSW) equations. The vortex dipoles are explicit travelling wave solutions to the QGSW equations, which are analogues of the classical circular vortex of Lamb and Chaplygin for the steady planar Euler equations. We establish a variational characterization of these vortex poles, which provides a basis for the stability result.
	\end{abstract}
	

	\section{Introduction and main results}
	
	In this paper,  we investigate the quasi-geostrophic shallow-water equation which is a nonlinear and nonlocal transport equation generalizing the two-dimensional Euler equations and used to describe large-scale motion for the atmosphere and the ocean circulation.

	\subsection{The quasi-geostrophic shallow-water equations}

	The quasi-geostrophic shallow water (QGSW) equations are derived asymptotically  from the rotating shallow-water equations, in the limit of rapid rotation and weak variations of the free surface \cite{vallis2017atmospheric}, which are given by
	\begin{equation}\label{1.1}
		\left\{ \begin{array}{l}
			\partial _{t}q+v\cdot \nabla q=0, \quad   (t,x)\in \mathbb{R} _+\times \mathbb{R} ^2,\\
			v=\nabla ^{\bot}\psi,\quad  \nabla^{\perp}=(\partial_{2}, -\partial_{1}),\\
			\psi =\left( -\Delta +\varepsilon ^2 \right) ^{-1}q,\\
			q_{\mid {t}=0}=q_0,\\
		\end{array} \right.
	\end{equation}
	where $ v $ is the velocity field, $ q$ is the `potential' vorticity, $ \psi $  is the stream function, and $ \varep \ge 0$ is a parameter.
	
	When the parameter $ \ep=0 $, we recover the two-dimensional Euler equations. The QGSW equations are a generalisation of the Euler equations and contain an additional parameter $\ep$.  The parameter $\ep$ is known as the inverse ‘Rossby deformation length’, which is a natural length scale arising from a balance between rotation and stratification.

	\subsection{The Lamb dipole} \label{sect1.2}
	The Lamb dipole is a special translational vortex pair, which has  a steady translating structure with opposite-sign vorticity of compact support in a circular disk. Translating vortex pairs are theoretical models of coherent vortex structures in large-scale geophysical flows; see \cite{flor1994experimental, van1989dipole}. Let us assume that a travelling wave solution is of the form
	\begin{align*}
		v\left( x,t \right) &=u\left( x+u_{\infty}t \right) -u_{\infty},
		\\
		q\left( x,t \right) &=\omega \left( x+u_{\infty}t \right) ,
	\end{align*}
	with a constant velocity $ u_{\infty}\in \mathbb{R}^{2} $ at space infinity. Vortex pairs are pairs of compactly supported dipoles, symmetrically placed with opposite signs, translating in one direction. Without loss of generality, we may assume that $ u_{\infty}=(-W, 0),\ W>0 $ by rotation invariance of \eqref{1.1}. Substituting $ (v, q) $ into equation \eqref{1.1}, we obtain the steady QGSW equations for $(u, \omega)$ in the half plane $ \Pi:=\{x=(x_{1}, x_{2})\in \mathbb{R}^{2}\,\mid\,x_{2}>0 \} $:
	\begin{equation}\label{1.33}
		\begin{split}
			u\cdot \nabla \omega =0, \quad  &\mathrm{in} \, \, \Pi ,
			\\
			u\rightarrow u_{\infty} \quad  &\mathrm{as} \, \, |x|\rightarrow \infty .
		\end{split}
	\end{equation}
	In 1906, Lamb \cite{Lamb1906} noted an explicit vortex pair solution to the two-dimensional Euler equaitons (i.e., $\ep=0$ in \eqref{1.1}), a solution $ \omega _{\mathrm{C}}=\lambda\left(\Psi_{\mathrm{C}}(x)-W x_{2}\right)_+,  u_{\mathrm{C}}=\nabla^{\perp}\Psi_{\mathrm{C}} -W\mathrm{e}_1,  0<\la<\infty$, of the form in the polar coordinate $(r, \theta)$
	\begin{equation}\label{Lamb}
		\Psi_{\mathrm{C}}(x)= \begin{cases}\left(A_{\mathrm{C}} J_{1}\left(\lambda^{1 / 2} r\right)+Wr\right) \sin \theta, & r \leq a, \\ \frac{a^{2}}{r} \sin \theta, & r>a,\end{cases}
	\end{equation}
	with the constants
	$$
	A_{\mathrm{C}}=-\frac{2 W}{\lambda^{1 / 2} J_{0}\left(c_{0}\right)}, \quad a=c_{0} \lambda^{-1 / 2},
	$$
	where $J_{m}(r)$ is the $m$-th order Bessel function of the first kind and the constant $c_{0}$ is the first zero point of $J_{1}$, i.e., $J_{1}\left(c_{0}\right)=0, c_{0}=3.8317 \cdots, J_{0}\left(c_{0}\right)<0$, $f_{+}$ denotes the positive part of $f$, and  $\mathrm{e}_1=(1,0)$. This explicit solution is indeed a special case of non-symmetric	Chaplygin	dipoles, independently founded by S. A. Chaplygin	\cite{chaplygin1903one, chaplygin2007one, meleshko1994chaplygin}. So it is now generally referred to as the Lamb dipole or Chaplygin–Lamb dipole. The stream function $ \Psi_{\mathrm{C}} $ satisfies the following elliptic equation
	\begin{equation}\label{add1}
		\left\{\begin{array}{l}
			-\Delta \Psi=\lambda\left(\Psi-W x_{2}\right)_+,\quad \quad  \text { in } \,  \Pi, \\
			\Psi \rightarrow 0 \text { as } r \rightarrow \infty, \quad \Psi=0, \, \, \, \text { on }  \, \partial \Pi, \\
			\Psi\left(x_{1}, x_{2}\right)=-\Psi\left(x_{1},-x_{2}\right), \quad \forall\, x \in \mathbb{R}^{2}.
		\end{array}\right.
	\end{equation}
	The Lamb dipole $\omega_{\mathrm{C}}$ has the form
	\begin{equation*}
		\omega_{\mathrm{C}}\left(x_{1}, x_{2}\right)=-\omega_{\mathrm{C}}\left(x_{1},-x_{2}\right)=\lambda\left(\Psi_{\mathrm{C}}(x)-W x_{2}\right)_+, \quad \forall\, x \in \Pi .
	\end{equation*}
	In 1996, Burton \cite{burton1996uniqueness} proved that $\Psi_{\mathrm{C}}$ is the unique solution to \eqref{add1} when viewed in a natural weak formulation by using the method of moving planes. Very recently, Abe and Choi \cite{abe2022stability} established nonlinear orbital stability of the Lamb dipole $\omega_{\mathrm{C}}$. For some numerical and experimental studies on stability, see \cite{flor1994experimental, van1998viscous}.
	
	In the present paper, we are interested in the Lamb dipole for the QGSW equations \eqref{1.1} with $\ep>0$. Without loss of generality, we shall restrict our attention to the case $\ep=1$. Let $ \Psi =(-\Delta+Id)^{-1}\omega $ and $\mathrm{e}_1=(1,0)$, then \eqref{1.33} can be rewritten as
	\begin{equation*}
		\left( \nabla ^{\bot}\Psi -W\mathrm{e}_1 \right) \cdot \nabla \omega =0,
	\end{equation*}
	which is equivalent to
	\begin{equation}\label{1.6}
		\nabla ^{\bot}\left( \Psi -Wx_2 \right) \cdot \nabla \omega =0.
	\end{equation}
	As remarked by V. I. Arnol'd \cite{arnol2013mathematical}, a natural way of obtaining solutions to the stationary
	problem \eqref{1.6} is to impose  that  $ \Psi -Wx_2 $ and $ \omega  $ are (locally) functional dependent. Inspired by \eqref{add1}, we assume that
	\begin{equation*}
		\omega =\lambda \left( \Psi -Wx_2 \right) _+\,\,\mathrm{in} \, \Pi
	\end{equation*}
	for some constant $\lambda$. The problem is thus transformed into finding a solution to the following problem
	\begin{equation}\label{1.7}
		\left\{\begin{array}{l}
			-\Delta \Psi+\Psi=\lambda\left(\Psi-W x_{2}\right)_{+},\quad \text {in } \,  \Pi, \\
			\Psi \rightarrow 0 \text { as } r \rightarrow \infty, \quad \Psi=0, \quad  \text { on } \partial \Pi, \\
			\Psi\left(x_{1}, x_{2}\right)=-\Psi\left(x_{1},-x_{2}\right), \quad  \forall\, x \in \R ^{2}.
		\end{array}\right.
	\end{equation}
	A solution of \eqref{1.7} can be easily found by using the separation of variables method. Indeed,
	let $1<\lambda<\infty$ and
	\begin{equation}\label{1.8}
		\Psi _L(x)=\begin{cases}
			\left( A_LJ_1((\lambda -1)^{1/2}r)+\frac{W\lambda}{\lambda -1}r \right) \sin \theta , &		r\le a,\\
			\frac{Wa}{K_1(a)}K_1(r)\sin \theta , &		r>a,\\
		\end{cases}
	\end{equation}
	where $ J_{1}(r) $ is the Bessel function of the first kind of order one,  $ K_{1} (r) $ is  the modified Bessel function of the second kind of order one,
	\begin{equation*}
		A_{L}=-\frac{Wa}{\la-1}\cdot \frac{1}{J_{1}((\la-1)^{1/2}a)},
	\end{equation*}
	and  $a$ be the smallest positive solution satisfying \eqref{1.99}
	\begin{equation}\label{1.99}
		a\left(\frac{K_{1}'(a)}{K_{1}(a)}+\frac{1}{(\la-1)^{1/2}}\cdot\frac{J_{1}'((\la-1)^{1/2}a)}{J_{1}((\la-1)^{1/2}a)}\right)=\frac{\la}{\la-1}.
	\end{equation}
	Then $\Psi_{L}$ is a desired solution of \eqref{1.7}. Moreover, $ \omega _{L}=\lambda \left( \Psi_L -Wx_2 \right) _+,  u_{L}=\nabla^{\perp}\Psi_{L}-W\mathrm{e}_1$ is an explicit solution to \eqref{1.33}. Its vorticity is positive inside a semicircular region, while outside this region the flow is irrotational. In conjunction with its reflection in the $x_1$-axis, this flow constitutes a circular vortex. We shall call this solution the Lamb dipole to the QGSW equations. It seems that limited work has been done for the Lamb dipole to the QGSW equations. There are some analytical and numerical studies of the vortex patch solution to the QGSW equations. Polvani \cite{polvani1988geostrophic} and Polvani, Zabusky and Flierl \cite{polvani1989two} computed the generalizations of Kirchhoff ellipses under various values of $\ep$,  including doubly-connected patches and multi-layer flows. Later, Plotka and Dritschel \cite{plotka_quasi-geostrophic_2012} numerically studied the equilibrium form and stability of the rotating simply-connected vortex patches for the QGSW equations. Very recently,  Dritschel, Hmidi and Renault \cite{dritschel_imperfect_2019}  investigated  both analytically and numerically the bifurcation diagram of simply-connected rotating vortex patch equilibria for the QGSW equations.

	The main purpose of this paper is to study the dynamical stability of the Lamb dipole for the QGSW equations. More precisely, we will establish the nonlinear orbital stability of the Lamb dipole $\omega _{L}$.

	\subsection{The main result}
	
	Similar to Burton \cite{burton_nonlinear_2013}, we introduce the following  $L^{p}$-regular solution:	
	\begin{definition}\label{def:1.3}
		For the function $\zeta \in L_{l o c}^{\infty}\left([0, \infty), L^{1}\left( \mathbb{R} ^{2}\right)\right) \cap L_{l o c}^{\infty}\left([0, \infty), L^{p}\left( \mathbb{R} ^{2}\right)\right)$ is called a $L^{p}$-regular solution of \eqref{1.1},   if $\zeta$ satisfies (\eqref{1.1}  in the sense of distributions, such that $E(\zeta(t, \cdot)), I(\zeta(t, \cdot))$ and $\|\zeta(t, \cdot)\|_{L^s}$ for $1 \leq s \leq p$ are constant for $t \in[0, \infty)$.
		Moreover, if $\zeta_{0}$ is non-negative and odd symmetric in $x_{2}$, then we require that $\zeta(t, \cdot)$ is also non-negative and odd symmetric in $x_{2}$.
	\end{definition}
	
	Roughly speaking, the $ L^{p}$-regular solution is a weak solution of \eqref{1.1}, whose kinetic energy, impulse, and $L^{s} $ norm are conserved when $ 1\le s\le p $. This is true for sufficiently smooth solutions.
	In the sequel, we identify a function $\zeta$ in $\Pi$ with an odd extension to $\mathbb{R}^2$ for the $x_2$-variable, i.e.,  $\zeta(x_1, x_2) = -\zeta(x_1, -x_2)$. We shall denote $\|\zeta\|_{L^1\cap L^2}:=\|\zeta\|_{1}+\|\zeta\|_2$.
	
	We have the following stability result:
	\begin{theorem}\label{thm1.4}
		The Lamb dipole $\omega_{L}$ is orbitally stable in the sense that for any $\varepsilon>0$, there exists $\delta>0$ such that for any non-negative function $\zeta_{0} \in L^{1} \cap L^{2}(\Pi)$ and
		\begin{equation*}
			\inf _{c \in \mathbb{R} }\left\{\left\|\zeta_{0}-\omega_{L}\left(\cdot+c \mathrm{e} _{1}\right)\right\|_{L^1\cap L^2}+\left\|x_{2}\left(\zeta_{0}-\omega_{L}\left(\cdot+c\mathrm{e} _{1}\right)\right)\right\|_{L^1}\right\} \leq \delta,
		\end{equation*}
		if there exists a $L^{2}$-regular solution $\zeta(t)$ with initial data ${\zeta_{0}}$, then
		\begin{equation*}
			\inf _{c \in \mathbb{R} }\left\{\left\|\zeta(t)-\omega_{L}\left(\cdot+c\mathrm{e} _{1}\right)\right\|_{L^1\cap L^2}+\left\|x_{2}\left(\zeta(t)-\omega_{L}\left(\cdot+c\mathrm{e} _{1}\right)\right)\right\|_{L^1}\right\} \leq \varepsilon, \quad \forall\, t \in[0, \infty).
		\end{equation*}
		
	\end{theorem}

	This paper is organized as follows. In Section \ref{sec1.5}, we provide a variational formulation for the Lamb dipole $\omega_L$. In Section \ref{sect2} we establish the existence of maximizers. The uniqueness of maximizers is proved in Section \ref{sect3}. Section \ref{sect4} is devoted to establishing the orbital stability in Theorem \ref{thm1.4}.

	\section{Variational formulation}\label{sec1.5}
	We shall use Arnol'd's idea \cite{arnol2013mathematical} (see also \cite{abe2022stability, cao_existence_2022, choi2020stability}) to establish the nonlinear stability. The key idea is to give a variational characterization of the Lamb dipole $\omega _{L}$. Since the desired flows are odd symmetric about the $x_{1}$-axis, we can restrict our attention henceforth to the upper half-plane $ \Pi $.  Let $ \bar{x}=(x_{1}, -x_{2}) $ be the reflection of $ x $ in	the $ x_{1} $-axis. Denote
	\begin{equation}\label{1.9}
		G_{\Pi}\left( x,y \right) =G\left( x,y \right) -G\left( \bar{x},y \right), \quad  \forall \,x,y\in \Pi ,
	\end{equation}
	where $G(x, y)=G(|x-y|)$ is the fundamental solution of the Bessel operator $-\Delta+Id$. Define
	\begin{equation}\label{1.10}
		\mathcal{G} \omega \left( x \right) =\int_{\Pi}{G_{\Pi}\left( x,y \right) \omega \left( y \right) \mathrm{d}y, \quad x\in \Pi .}
	\end{equation}
	We introduced the kinetic energy of the fluid as follows
	\begin{equation*}
		E\left( \omega \right) =\frac{1}{2}\int_{\Pi}{\omega \left( x \right) \mathcal{G} \omega \left( x \right) \mathrm{d} x,}
	\end{equation*}
	and its impulse
	\begin{equation*}
		I\left( \omega \right) =\int_{\mathbb{R} ^2}{x_2\omega \left( x \right) \mathrm{d}x.}
	\end{equation*}
	Let $\lambda>1$, $\mu>0$ and $\nu>0$. We introduce the following space of admissible functions
	\begin{equation*}
		\mathcal{A} _{\mu ,\nu}:=\left\{ \omega \in L^2\left( \Pi \right) \mid \omega \ge 0, \int_{\Pi}{x_2\omega(x)\, \mathrm{d}x = \mu , \int_{\Pi}{\omega}(x)\,\mathrm{d}x\le \nu \,\,} \right\},
	\end{equation*}
	and the energy functional  $ \mathcal{E} _{\la}$ corresponding to the flows
	\begin{equation*}
		\mathcal{E} _{\lambda}\left( \omega \right) =E\left( \omega \right) -\frac{1}{2\lambda}\int_{\Pi}{\omega ^2\mathrm{d} x,  \,\, \, \omega \in \mathcal{A} _{\mu ,\nu}.}
	\end{equation*}
	We will consider the maximization of the energy functional $ \mathcal{E}_{\la} $ relative to $ \mathcal{A} _{\mu ,\nu} $.  Set
	\begin{equation}\label{1.11}
		S_{\mu ,\nu ,\lambda}:=\mathop {\mathrm{sup}} \limits_{\omega \in \mathcal{A} _{\mu ,\nu} }\mathcal{E} _{\lambda}\left( \omega \right),
	\end{equation}
	and
	\begin{equation}\label{1.12}
		\Sigma_{\mu, \nu, \lambda} :=\left\{ \omega \in \mathcal{A} _{\mu ,\nu} \mid \mathcal{E} _{\lambda}\left( \omega \right) =S_{\mu ,\nu ,\lambda} \right\}.
	\end{equation}
	Recall that $$\omega_L=\omega_L^{\lambda, W}=\lambda \left( \Psi_L^{\lambda, W} -Wx_2 \right) _+,$$ where $\Psi_L^{\lambda, W}$ is given by \eqref{1.8}. We will show that the Lamb dipole $\omega _{L}^{\lambda, W}$ can be re-obtained via the maximization problem \eqref{1.11} by appropriately choosing the parameters. More precisely, we have (see also Corollary \ref{add-cor} below)
	
	\begin{proposition}\label{thm1.2}
		Let $\lambda>1$ and $\mu>0$ be given. Then there exists $\nu_0>0$, such that if $\nu\ge \nu_0$, then
		\begin{equation*}
			\Sigma_{\mu, \nu, \lambda}=\left\{\omega_{L}^{\lambda, W}\left(\cdot+c \mathrm{e}_{1}\right) \mid c \in \mathbb{R} \right\},
		\end{equation*}
		where $W=\mu/I(\omega_L^{\lambda, 1})$.
	\end{proposition}

	\section{Existence of Maximizers}\label{sect2}
	In this section, we prove the existence of maximizers for $\mathcal{E} _{\lambda}$ over $\mathcal{A} _{\mu ,\nu}$. We first give some basic estimates that will be used frequently later. In what follows, the symbol $C$ denotes a general positive constant that may change from line to line.
	We have the following basic estimate:
	\begin{equation}\label{ess}
		G(x,y) =\begin{cases}
			C_0\left( \ln \frac{2}{|x-y|}+O(1)  \right),\quad&\mathrm{if} \ \   |x-y|\le 2,\\
			O\Big(e^{-|x-y|/2}\Big),  &\mathrm{if} \ \  |x-y|>2,\\
		\end{cases}
	\end{equation}
	where $C_0$ is a positive number.
	
	\begin{lemma} \label{lm2.1}
		There exists a positive constant $ C $ such that if $ 0\le\ \omega \in L^{1}(\Pi)\cap L^{2}(\Pi)$,  then
		\begin{equation}\label{2.1}
			\|\mathcal{G}\omega\|_\infty \le C\|\omega \|_{1}^{1/2}\|\omega \|_{2}^{1/2},
		\end{equation}
		and
		\begin{equation}\label{2.3}
			E(\omega) \le C\|\omega \|_{1}^{3/2}\|\omega \|_{2}^{1/2}.
		\end{equation}
	\end{lemma}
	\begin{proof}
		Let us first prove \eqref{2.1}. By Hölder's inequality, we have
		\[
		\int_{\Pi}G_{\Pi}(x,y)\omega(y) \mathrm{d} y
		\le  C \|\omega \|_{4/3}\le C\|\omega \|_{1}^{1/2}\|\omega \|_{2}^{1/2},\quad\forall\, x\in \Pi.
		\]
		By the definition of $ E$ and  \eqref{2.1}, we get
		\[
		E(\omega )\le C \|\mathcal{G}\omega\|_\infty \|\omega \|_{1}\le C\|\omega \|_{1}^{3/2}\|\omega \|_{2}^{1/2}.
		\]
		The proof is thus complete.
	\end{proof}

	\begin{lemma}\label{lm2.22}
		Suppose that  $ 0\le \omega \in L^{1}(\Pi)\cap L^{2}(\Pi)$, we have
		\begin{equation}\label{2.2}
			\mathcal{G}\omega(x)\rightarrow0 \quad \text{as}\ |x|\rightarrow\infty.
		\end{equation}
	\end{lemma}
	\begin{proof}
		
		For $|x|$ large, by \eqref{ess} and \eqref{2.1} we have
		\[
		\begin{split}0\le\mathcal{G}\omega(x) & \le\int_{|y|\le|x|/2}G_{\Pi}(x,y)\omega(y)\mathrm{d} y+\int_{|y|\ge|x|/2}G_{\Pi}(x,y)\omega(y)\mathrm{d} y\\
			& \le C\Big(e^{-|x|/4}\|\omega\|_{1}+\|\omega1_{\Pi\setminus B_{|x|/2}(0)}\|_{1}+\|\omega1_{\Pi\setminus B_{|x|/2}(0)}\|_{2}\Big)\\
			& =o(1),
		\end{split}
		\]
		which implies \eqref{2.2} and completes the proof.
	\end{proof}
	
	Since the energy $\mathcal{E}_{\lambda}$ is invariant under translations in the $x_{1}$-direction, to control maximizers, we shall  take the Steiner symmetrization in the $x_{1}$-variable.
	
	We have the following result, whose proof is quite similar to that in \cite{fraenkel1974global,turkington1983steady1, turkington1983steady2} and so is omitted.
	
	\begin{lemma} \label{lm2.2}
		For $ \omega \ge 0 $ satisfying $ \omega \in L^{1}\cap L^{2}(\Pi) $ and $ x_{2}\omega \in L^{1}(\Pi) $,  there exists $ \omega ^{*} \ge 0$ such that
		\begin{equation}
			\begin{split} & \omega^{*}(x_{1},x_{2})=\omega^{*}(-x_{1},x_{2}),\\
				& \omega^{*}(x_{1},x_{2})\ \text{is non-increasing for}\ x_{1}>0
			\end{split}
		\end{equation}
		and
		\begin{align*}
			& \|\omega^{*}\|_{s}=\|\omega\|_{s},\ \ \forall\, s\in[1,2],\\
			& \|x_{2}\omega^{*}\|_{1}=\|x_{2}\omega^{*}\|_{1},\\
			& E(\omega^{*})\ge E(\omega).
		\end{align*}
	\end{lemma}
	
	For a Steiner symmetric function, we have the following estimate:
	\begin{lemma} \label{lm2.3}
		There exists a positive constant $ C $ such that if $ 0\le \omega \in L^{1}(\Pi)\cap L^{2}(\Pi)$ is Steiner symmetric	in	the $ x_{1} $-variable, then
		\begin{equation}\label{2.5}
			\mathcal{G}\omega(x)\le C\Big(|x_{1}|^{-3/8}\|\omega \|_{1 }^{1/2}\|\omega \|_{2 }^{1/2}+e^{-\frac{\sqrt{|x_1|}}{2}}\|\omega\|_{1}\Big)
		\end{equation}
		for any $x=(x_1,x_2)\in \Pi$ with $|x_1|>4$.
	\end{lemma}
	\begin{proof}
		Let $x=(x_1,x_2)\in \Pi$ satisfy $|x_1|>4$. Define
		\begin{equation*}
			\omega_{1}(y)=\begin{cases}
				\omega(y),\quad&\text{ if }\left|y_{1}-x_{1}\right|<\sqrt{\left|x_{1}\right|},\\
				0, & \text{ if }\left|y_{1}-x_{1}\right|\geq\sqrt{\left|x_{1}\right|}.
			\end{cases}
		\end{equation*}
		Using Eq. (2.11) in \cite{burton_steady_1988}, we have
		\begin{equation*}
			\|\omega_{1}\|_{p}\le |x_{1}|^{-\frac{1}{2p}}\|\omega\|_{p},\quad 1\le p \le \infty.
		\end{equation*}
		Hence, by \eqref{2.1}, we have
		\begin{equation}\label{2.6}
			\begin{split}
				\mathcal{G}\omega_{1}(x)&\le C\|\omega _{1}\|_{1 }^{1/2}\|\omega _{1}\|_{2 }^{1/2}\\
				&\le C|x_{1}|^{-3/8}\|\omega \|_{1 }^{1/2}\|\omega \|_{2 }^{1/2}.\\
			\end{split}
		\end{equation}
		Letting $ \omega_{2} =\omega-\omega_{1}$, we have
		\begin{equation}\label{2.7}
			\mathcal{G}\omega_{2}(x)\le Ce^{-\frac{\sqrt{|x_1|}}{2}}\|\omega\|_{1}.
		\end{equation}
		Combining \eqref{2.6} and \eqref{2.7},  we get  \eqref{2.5}.
	\end{proof}
	
	Set
	\begin{equation*}
		\varrho(\lambda)=\frac{1}{I(\omega_L^{\lambda, 1})}\int_{\Pi} \omega_L^{\lambda, 1}(x)\,\mathrm{d} x.
	\end{equation*}
	Note that $\omega_L^{\lambda, W}=W\omega_L^{\lambda, 1}$. We have the following result concerning the supremum value.
	\begin{lemma}\label{lm2.4}
		If $\mu \varrho(\lambda)\le \nu$, then
		\begin{equation}\label{2.8}
			0< S_{\mu ,\nu ,\lambda}\le C.
		\end{equation}
	\end{lemma}
	\begin{proof}
		By \eqref{2.3} and Young's inequality, we have for $ \omega \in \mathcal{A}_{\mu, \nu}$
		\begin{equation*}
			\begin{split}
				\mathcal{E}_{\lambda}(\omega )&=E(\omega)-\frac{1}{2\la}\int_{\Pi}\omega^{2}\mathrm{d} x\\
				& \le C\|\omega \|_{1}^{3/2}\|\omega \|_{2}^{1/2}-\frac{1}{2\la}\int_{\Pi}\omega^{2}\mathrm{d} x\\
				&\le C\lambda^{1/3}\|\omega \|_{1}^2\le C.
			\end{split}
		\end{equation*}
		On the other hand, since $\omega_L^{\lambda, \tilde{W}}$ with $\tilde{W}=\mu/I(\omega_L^{\lambda, 1})$ belongs to $\mathcal{A}_{\mu,\nu}$, so
		\begin{align*}
			\mathcal{E} _{\lambda}\big( \omega_L^{\lambda, \tilde{W}} \big) &=\frac{1}{2}\int_{\Pi}{\lambda \big( \Psi_L^{\lambda, \tilde{W}}-\tilde{W}x_2 \big) _+\Psi_L^{\lambda, \tilde{W}}\mathrm{d} x-\frac{1}{2\lambda}\int_{\Pi}{\Big( \lambda \big( \Psi_L^{\lambda, \tilde{W}}-\tilde{W}x_2 \big) _+ \Big)^{2}}}\mathrm{d} x
			\\&
			=\frac{1}{2}\int_{\Pi}{\lambda}\big( \Psi_L^{\lambda, \tilde{W}}-\tilde{W}x_2 \big) _+\Big( \Psi_L^{\lambda, \tilde{W}}-\big[ \Psi_L^{\lambda, \tilde{W}}-\tilde{W}x_2 \big] _+ \Big) \mathrm{d} x
			>0.
		\end{align*}
		Therefore $S_{\mu ,\nu ,\lambda}\ge \mathcal{E} _{\lambda}\big( \omega_L^{\lambda, \tilde{W}} \big)>0$ and the proof is thus complete.
	\end{proof}
	In the sequel we shall assume that $\mu \varrho(\lambda)\le \nu$. Having made all the preparation, we are now able to show the existence of maximizers.
	\begin{lemma}\label{lm2.5}
		It holds $\Sigma _{\mu ,\nu ,\lambda} \not =\emptyset$. In addition, each $\omega\in \Sigma _{\mu ,\nu ,\lambda}$ satisfies $\displaystyle \int_{\Pi} x_{2}\omega(x) \mathrm{d} x=\mu$.
	\end{lemma}
	\begin{proof}
		Let $ \{\omega_{j}\} _{j=1}^{\infty}\subset \mathcal{A}_{\mu, \nu}$ be a maximizing sequence. By Lemma \ref{lm2.4}, we may assume that $ \mathcal{E}_{\lambda}(\omega_{j})\ge 0 $ for all large $ j $.
		Using the definition of $ \mathcal{E}_{\la} $ and \eqref{2.3}, we have
		\begin{equation*}
			\|\omega_{j}\|_{2}^{2}\le 2\la \big(E(\omega_{j})-\mathcal{E}_{\lambda}(\omega_{j})\big)\le 2\la E(\omg_{j})\le
			C\|\omega_{j} \|_{1}^{3/2}\|\omega_{j} \|_{2}^{1/2} \le C\|\omg_{j}\|_{2}^{1/2}.
		\end{equation*}
		Hence $ \|\omg_{j}\|_{2} $ is bounded by a constant independent of $ j $. We are going to show the convergence of energy.
		According to Lemma \ref{lm2.2},  we may assume that $ \omg_{j} $ is Steiner symmetric by replacing $ \omg_{j} $ with its Steiner symmetrisation. We  assume $ \omg_{j}\rightarrow \omg $ weekly in $ L^{2}(\Pi) $ as $ j\rightarrow \infty $ by passing to a sub-sequence if necessary (still denoted by $ \{\omg_{j}\}_{j=1}^{\infty} $ ). It is easy to verify that
		\begin{equation*}
			\int_{\Pi} x_{2}\omg \mathrm{d} x\le \mu \quad \text{and}\quad  \int_{\Pi}\omega \mathrm{d} x \le \nu.
		\end{equation*}
		
		On the one hand, by Lemmas \ref{lm2.1} and \ref{lm2.3}, we have
		\begin{align*}
			2E(\omega_{j})&=\int_{\Pi}\int_{\Pi}\omega_{j}(x)G_{\Pi}(x,y)\omega_{j}(y)\mathrm{d} x\mathrm{d} y\\
			&\le \int_{|x_{1}|<R, 0<x_{2}<R}\int_{|y_{1}|<R, 0<x_{2}<R}\omega_{j}(x)G_{\Pi}(x,y)\omega_{j}(y)\mathrm{d} x\mathrm{d} y\\
			&\ \ \ \ \ +2\int_{x_{2}\ge R}\omega_{j}(x) \mathcal{G}\omega_j(x)\mathrm{d} x         +2\int_{|x_{1}|\ge R}\omega_{j}(x) \mathcal{G}\omega_j(x)\mathrm{d} x \\
			&\le \int_{|x_{1}|<R, 0<x_{2}<R}\int_{|y_{1}|<R, 0<x_{2}<R}\omega_{j}(x)G_{\Pi}(x,y)\omega_{j}(y)\mathrm{d} x \mathrm{d}y\\
			& \ \ \ \ \ \ +C\Big(R^{-3/8}\|\omega _{j}\|_{1}^\frac{3}{2}\|\omega _{j}\|_{2 }^{1/2}+ e^{-\frac{\sqrt{R}}{2}}\|\omega_{j}\|_{1}^{2}\Big)+2R^{-1}\|\mathcal{G}\omega_j\|_\infty\int_{\Pi}x_2\omega_{j}(x)\mathrm{d} x \\
			&\le \int_{|x_{1}|<R, 0<x_{2}<R}\int_{|y_{1}|<R, 0<x_{2}<R}\omega_{j}(x)G_{\Pi}(x,y)\omega_{j}(y)\mathrm{d}x\mathrm{d}y+C\Big(R^{-3/8}+e^{-\frac{\sqrt{R}}{2}}+R^{-1}  \Big).
		\end{align*}
		Thanks to $ G_{\Pi}(x, y) \in L_{loc}^{2}(\overline{\Pi}\times\overline{\Pi}) $, we get
		\begin{equation*}
			\limsup_{j\rightarrow\infty}E(\omega_{j})\le E(\omega)
		\end{equation*}
		by first letting $ j\rightarrow\infty $ and then $ R\rightarrow\infty $.
		
		On the other hand, we have
		\begin{equation*}
			2E(\omega_{j})=\int_{\Pi}\omega_{j}\mathcal{G}\omega_{j}\mathrm{d} x\ge \int_{|x_{1}|<R, 0<x_{2}<R}\int_{|y_{1}|<R, 0<x_{2}<R}\omega_{j}(x)\mathcal{G}\omega_{j}(y)\mathrm{d} x\mathrm{d} y ,
		\end{equation*}
		it implies that
		\begin{equation*}
			\liminf_{j\rightarrow\infty}E(\omega_{j})\ge E(\omega)
		\end{equation*}
		by first letting $ j\rightarrow\infty $ and then $ R\rightarrow\infty $.
		
		Hence, we conclude that
		\begin{equation*}
			\lim_{j\rightarrow\infty}E(\omega_{j})=E(\omega).
		\end{equation*}
		and
		\begin{equation*}
			\mathcal{E}_{\la}(\omega)=E(\omega)-\frac{1}{2\la}\int_{\Pi}\omega^{2}\mathrm{d} x\ge \lim_{j\rightarrow\infty}E(\omega_{j})-\frac{1}{2\la} \cdot\liminf_{j\rightarrow\infty}\int_{\Pi}\omega_{j}^{2}\mathrm{d} x=S_{\mu ,\nu ,\lambda}.
		\end{equation*}
		
		We now check that $\displaystyle \int_{\Pi} x_{2}\omega \mathrm{d} x=\mu$. Indeed, suppose not, then there exists some $ \tau>0 $ such that
		\begin{equation*}
			\omega_{\tau}(x_{1}, x_{2}):=\begin{cases}
				\omega(x_{1}, x_{2}-\tau),\quad &\text{if}\ \ x_{2}>\tau,\\
				0 ,& \text{if}\ \ x_{2}\le \tau,\\
			\end{cases}
		\end{equation*}
		belongs to $ \mathcal{A}_{\mu, \nu} $. By virtue of the facts that $ G_{\Pi}(x, y)=G(|x-y|)-G(|\bar{x}-y|) $ and $ G(s) $ is strictly decreasing for $s>0$, we check that
		\begin{equation*}
			{S} _{\mu ,\nu ,\lambda}=\mathcal{E}_{\la}(\omega)<\mathcal{E}_{\la}(\omega_{\tau})\le {S} _{\mu ,\nu ,\lambda}.
		\end{equation*}
		This is a contradiction and the proof is thus complete.
	\end{proof}
	
	From the proof of Lemma \ref{lm2.5}, we can obtain the monotonicity of $ S_{\mu ,\nu ,\lambda} $ with respect to $ \mu $.
	\begin{lemma}\label{lm2.6}
		If $0<\mu_1<\mu_2$, then $S_{\mu_1 ,\nu ,\lambda}<S_{\mu_2 ,\nu ,\lambda}$.
	\end{lemma}

	\section{Uniqueness of Maximizers}\label{sect3}
	In the preceding section, we have proved the existence of maximizers for $ \mathcal{E}_{\la} $ over $\mathcal{A}_{\mu, \nu}$. In this section, we will establish the uniqueness of maximizers in the sense that any two maximizers differ by only a translation in the $x_1$-direction.
	\begin{lemma}\label{lm3.1}
		Each $ \omega\in \Sigma_{\mu ,\nu ,\lambda} $ satisfies
		\begin{equation}\label{3.1}
			\omega=\la(\mathcal{G}\omega-Wx_{2}-\gamma)_{+}
		\end{equation}
		for some constants $ W,\, \gamma\ge 0 $, uniquely determined by $ \omega $.
	\end{lemma}
	\begin{proof}
		By Lemma \ref{lm2.4}, $ {S}_{\mu ,\nu ,\lambda}>0 $. There exists a  constant $ \delta_{0}>0$ such that $ \text{meas}(\{\delta _{0}<\omega\})>0 $. We take functions $ h_{1}, h_{2}\in L^{\infty}(\Pi) $ with compact support and satisfying
		\begin{equation*}
			\begin{cases}
				\supp(h_{1}), \, \supp(h_{2}) \subset \{ \delta _{0}\le \omega \}, \\
				\int_{\Pi}h_{1}(x)\mathrm{d} x=1, \quad \int_{\Pi}x_{2}h_{1}(x)\mathrm{d} x=0, \\
				\int_{\Pi}h_{2}(x)\mathrm{d} x=0, \quad  \int_{\Pi}x_{2}h_{2}(x)\mathrm{d} x=1. \\
			\end{cases}
		\end{equation*}
		We take an  arbitrary $ \delta \in (0, \delta _{0}) $ and compactly supported $ h\in L^{\infty}(\Pi) $, $ h\ge 0 $ on $ \{0\le \omega \le \delta \} $. We consider the test functions
		\begin{equation*}
			\omega _{\ep} =\omega +\ep \eta,\ \ep>0,
		\end{equation*}
		where
		\begin{equation*}
			\eta=h-\Big(\int_{\Pi}h\mathrm{d} x\Big)h_{1}-\Big(\int_{\Pi}x_{2}h\mathrm{d} x\Big)h_{2}.
		\end{equation*}
		If $ \ep $ is small enough, one can verify that $ \omega _{\ep}\in \mathcal{A}_{\mu, \nu} $. Since $ \omega  $ is a maximizer,
		\begin{equation*}
			0\ge \frac{\mathrm{d}\mathcal{E}_\lambda(\omega _{\ep})}{\mathrm{d}\ep} \Big| _{\ep=0}=\int_{\Pi}\Big(\mathcal{G}\omega-\frac{1}{\la}\omega \Big)\eta \mathrm{d} x.
		\end{equation*}
		We define
		\begin{equation*}
			\gamma:=\int_{\Pi}\Big(\mathcal{G}\omega -\frac{1}{\la}\omega \Big)h_{1}\mathrm{d}x, \quad   W:=\int_{\Pi}\Big(\mathcal{G}\omega -\frac{1}{\la}\omega \Big)h_{2}\mathrm{d} x,
		\end{equation*}
		and
		\begin{equation*}
			\Psi:=\mathcal{G}\omega -Wx_{2}-\gamma.
		\end{equation*}
		Hence we get
		\begin{align*}
			0&\ge \int_{\Pi}\Big(\mathcal{G}\omega -\frac{1}{\la}\omega \Big)\eta \mathrm{d}x\\
			&=\int_{\Pi} \Big(\Psi-\frac{1}{\la}\Big)h\mathrm{d}x.
		\end{align*}
		Since  the arbitrariness of $ h $, we have
		\begin{align*}
			\begin{cases}
				\Psi -\frac{1}{\lambda}\omega =0, \quad  &\mathrm{on}\ \{\omega >\delta \},\\
				\Psi -\frac{1}{\lambda}\omega \le0, \quad &\mathrm{on}\ \{0\le \omega \le \delta \}.\\
			\end{cases}
		\end{align*}
		By letting $ \delta \rightarrow0 $, we obtain $ \omega =\la \Psi_{+} $.
		
		According to $ \int_{\Pi}\omega \mathrm{d} x\le \nu $, we can take a sequence $ \{x_{i}\}_{i=1}^{\infty} $ with $ x_{i}=(x_{1i}, x_{2 i}) $ , such that $ x_{1i}\rightarrow\infty$, $x_{2i}\rightarrow0 $ and $ \omega (x_{i})\rightarrow0 $ as $ i\rightarrow\infty $. By \eqref{2.2} in Lemma \ref{lm2.22}, we have
		\begin{equation*}
			\limsup_{n\rightarrow\infty}\left(\mathcal{G}\omega (x_{i})-Wx_{2i}-\gamma \right)\le 0.
		\end{equation*}
		Hence $ \gamma\ge 0 $. Similarly, we can take another sequence $ \{x_{j}\}_{j=1}^{\infty} $ with $ x_{j}=(x_{1j}, x_{2 j}) $, such that $ x_{1j}\rightarrow0$, $x_{2j}\rightarrow\infty$ and $ \omega (x_{j})\rightarrow0 $ as $ j\rightarrow\infty $. By \eqref{2.2} in Lemma \ref{lm2.22}, we have
		\begin{equation*}
			0=\lim_{j\rightarrow\infty}(\mathcal{G}\omega (x_{j})-Wx_{2j}-\gamma)_{+}=\lim_{j\rightarrow\infty}(-Wx_{2j}-\gamma)_{+},
		\end{equation*}
		which implies $ W\ge 0 $.

		Next, we show the uniqueness of $ W $ and $ \gamma $. Suppose \eqref{3.1} holds with $W_1,\, \gamma_{1}\ge 0$. Then
		\begin{equation*}
			\mathcal{G}\omega(x)-W_{1}x _{2}-\gamma _{1}=\mathcal{G}\omega (x)-Wx _{2}-\gamma,
		\end{equation*}
		for all 	$x\in \Pi$ satisfying $\omega(x)>0$. Then,
		\begin{equation*}
			(W _{1}-W)x _{2}=\gamma-\gamma _{1},
		\end{equation*}
		which implies $W _{1}=W$ and $\gamma _{1}=\gamma$.
	\end{proof}
	
	The following result shows that if $\nu$ is sufficiently large, then $ W>0 $ and $ \gamma=0 $.
	\begin{lemma}\label{lm3.2}
		Given $\lambda>1$ and $\mu>0$, there exists $\nu_0>\mu \varrho(\lambda)$ such that if $\nu\ge \nu_0$, then
		the constants $ W>0$, $\gamma=0 $ in Lemma \ref{lm3.1}.
	\end{lemma}
	\begin{proof}
		Let $\lambda>1$ and $\mu>0$ be fixed and $ \omega \in \Sigma_{\mu ,\nu ,\lambda} $, we start to prove $ \gamma=0 $ for all large $\nu$. Since
		\begin{equation*}
			\mu=\int_{\Pi}x_{2}\omega \mathrm{d} x\ge\frac{2\mu}{\nu}\int_{x_{2}\ge \frac{2\mu}{\nu}}\omega \mathrm{d} x,
		\end{equation*}
		we have \begin{equation}\label{eq3.2}
			\int_{x_{2}\ge\frac{2\mu}{\nu}}\omega \mathrm{d} x\le \frac{\nu}{2}.
		\end{equation}		
		By Lemma \ref{lm3.1}, $ \omega \le \la \mathcal{G}\omega$, so
		\begin{equation*}
			\int_{0<x_{2}<\frac{2\mu}{\nu}}\omega \mathrm{d} x\le \int_{\Pi}\int_{0<y_{2}<\frac{2\mu}{\nu}}\la G_{\Pi}(x, y)\omega (x) \mathrm{d} y\mathrm{d} x.
		\end{equation*}
		By \eqref{ess}, for $\nu$ large we have
		\begin{equation*}
			\int_{0<y_{2}<\frac{2\mu}{\nu}}{G}_{\Pi}(x, y)\mathrm{d} y=o(1),
		\end{equation*}
		uniformly with respect to $x$. Hence
		\begin{equation}\label{eq3.3}
			\int_{0<x_{2}<\frac{2\mu}{\nu}}\omega \mathrm{d} x=o(1)\nu
		\end{equation}
		for $\nu$ large. Combining \eqref{eq3.2} and \eqref{eq3.3}, we see that for all sufficiently large $\nu$, it holds
		\begin{equation*}
			\int_{\Pi}\omega \mathrm{d} x< \nu.
		\end{equation*}
		Hence, we can take
		\begin{equation*}
			\eta=h-\Big(\int_{\Pi}x _{2}h\mathrm{d} x\Big)h _{2}.
		\end{equation*}
		Consider the test functions  $\omega+\varepsilon\eta$. Proceeding as in the proof of Lemma \ref{lm3.1}, we can obtain
		\begin{equation*}
			\omega =\la(\mathcal{G}\omega -Wx_{2})_{+},
		\end{equation*}
		which implies $\gamma=0$.
		
		Now, we turn to prove $W>0$ for $\nu$ large. By ~\eqref{3.1}, we have
		\begin{align*}
			0&<\int_{\Pi}\omega  \mathcal{G}\omega\mathrm{d} x -\frac{1}{\la}\int_{\Pi}\omega^2\mathrm{d} x\\
			&=\int_{\Pi}\omega\mathcal{G}\omega\mathrm{d} x-\int_{\Pi}\omega(\mathcal{G}\omega -Wx _{2})_+\mathrm{d} x\\
			&\le \int_{\Pi}\omega \mathcal{G}\omega\mathrm{d} x-\int_{\Pi}\omega(\mathcal{G}\omega-Wx _{2} )\mathrm{d} x\\
			&=W\mu,
		\end{align*}
		which implies $W>0$. The proof of Lemma \ref{lm3.2} is thus finished.
	\end{proof}
	
	The following result shows that each maximizer has compact support in $\overline{\Pi}$. We denote by $BUC(\overline{ \Pi})$ the space of all bounded uniformly continuous functions in $\overline{\Pi}$ and by $C^{\alpha}(\overline{\Pi})$ the space of all H$\ddot{\text{o}}$lder continuous functions of exponent $0<\alpha<1$ in $\overline{\Pi}$. For an integer $k\ge 0$, $BUC^{k, \alpha}(\overline{\Pi})$ denotes the space of all $\phi \in BUC(\overline{ \Pi})$ such that $\partial_{x}^{l} \phi \in BUC(\overline{\Pi}) \cap C^{\alpha}(\overline{\Pi})$, for $|l| \leq k$.
	\begin{lemma}\label{lm3.3}
		For each $\omega \in \Sigma_{\mu ,\nu ,\lambda}$,  $ \supp(\omega  ) $ is a compact set in $ \overline{\Pi} $.
	\end{lemma}
	\begin{proof}
		Let $ \omega \in \Sigma_{\mu ,\nu ,\lambda}  $. By \eqref{3.1},  we have $ \supp(\omega) =\overline{\{ x \in \Pi ~\mid~\mathcal{G}\omega -Wx_{2}-\gamma>0\}}$ for $ W\ge 0 $ and $ \gamma\ge 0 $. If $ \gamma>0 $,  the conclusion follows easily from \eqref{2.2}. If $ \gamma =0 $, we must have $ W>0 $. By \eqref{3.1},  we have $ \supp(\omega) =\overline{\{ x \in \Pi ~\mid~\mathcal{G}\omega -Wx_{2}>0\}}$. It follows from $ \omega \in L^{1} \cap L^{2} $ that $ \nabla^{2} \mathcal{G}\omega \in L^{p} $,  $ p \in (1, 2) $ and $ \nabla \mathcal{G}\omega \in  L^{q}$,  $ 1/q=1/p-1/2 $.  By \eqref{2.2} and \eqref{3.1},   $ \mathcal{G}\omega  $  satisfies the following elliptic equation
		\begin{align*}
			\begin{cases}
				-\Delta \psi +\psi =\lambda(\psi -Wx_2)_+, \quad&\mathrm{in}\ \ \Pi ,\\
				\psi =0,&\mathrm{on}\ \ \partial \Pi ,\\
				\psi \rightarrow 0,&\mathrm{as}\ \ |x|\rightarrow \infty .\\
			\end{cases}
		\end{align*}
		By the Sobolev embedding,  we have $ \mathcal{G}\omega \in BUC^{2, \alpha} (\Pi)$. Since $ \mathcal{G}\omega (x_{1}, 0)=0 $  and
		\begin{equation*}
			\frac{\mathcal{G}\omega }{x_{2}}=\int_{0}^{1}(\partial_{2}\mathcal{G}\omega )(x_{1}, x_{2}s)\mathrm{d} s,
		\end{equation*}
		hence $ \mathcal{G}\omega /x_{2} \in BUC^{1, \alpha}(\bar{\Pi}) $. Using Hardy's inequality (\cite{maza2011sobolev}), we get
		\begin{equation*}
			\|\mathcal{G}\omega /{x_{2}}\|_{2}\le 2\|\nabla\mathcal{G}\omega \|_{2},
		\end{equation*}
		and hence $ \mathcal{G}\omega /x_{2} \in BUC(\Pi)\cap L^{2}(\Pi) $. It follows that
		\begin{equation*}
			\frac{\mathcal{G}\omega(x) }{x_{2}}\rightarrow0\quad \text{as}\ |x| \rightarrow \infty,
		\end{equation*}
		which implies that $ \supp(\omega ) $ is a compact set of $\overline{\Pi}$.
	\end{proof}
	
	Next, we consider positive solutions to the problem
	\begin{equation}\label{3.6}
		\begin{cases}
			-\Delta \psi +\psi =\lambda (\psi -Wx_2)_+, \quad &\mathrm{in}\ \ \Pi ,\\
			\psi =0,&		\mathrm{on} \  \partial \Pi ,\\
			\psi(x) \rightarrow 0,&		\mathrm{as}\ \  |x|\rightarrow \infty .\\
		\end{cases}
	\end{equation}

	\begin{lemma}\label{lm3.4}
		Let $ \psi \in BUC^{2, \alpha}(\overline{ \Pi})$,  $ 0<\alpha<1 $, be a positive solution of \eqref{3.6} for some $ W>0 $  and $ \la>1 $.
		Then $ \psi (x)=\psi_L(x+c\textbf{e}_{1}) $ for some $ c\in \mathbb{R} $,  where $\psi_{L}= \Psi_{L}$ and $ \Psi_{L} $ is defined by \eqref{1.8}.

	\end{lemma}
	
	\begin{proof}
		For $ y=(y', y_{4})\in \mathbb{R}^{4}$, $ y'=(y_{1},  y_{2}, y_{3})$, we set $ x_{1}=y_{4}$, $x_{2}=|y'| $ and
		\begin{equation}\label{3.7}
			\phi(y)= \frac{\psi(x_{1}, x_{2})}{x_{2}}.
		\end{equation}
		By a direct calculation,  we have
		\begin{align*}
			\begin{cases}
				-\Delta _y\phi +\phi =\lambda (\phi -W)_+,\quad&\mathrm{in} \ \ \mathbb{R} ^4,\\
				\phi \rightarrow 0,&		\mathrm{as} \  \ |y|\rightarrow \infty .\\
			\end{cases}
		\end{align*}
		Thus  $ \phi $ satisfies the integral equation
		\begin{equation}\label{3.8}
			\phi(x)=\int_{\mathbb{R}^{4}}G_4(x-y)\la(\phi(y)-W)_{+}\mathrm{d} y.
		\end{equation}
		where $ G_4 $ is the fundamental solution of the Bessel equation in	$ \mathbb{R}^{4}$.

		Since $\phi$ is continuous and the support of $(\phi(y)-W)_{+}$ is compact, one can apply the standard method of moving planes in the integral form  to deduce that $\phi$ is radially symmetric with respect to some point $y^{0}=\left(0, c\right)\in \mathbb{R}^{4}$, see \cite{chen, gidas1979symmetry} for more details.
		
		Hence $ \varphi(y)=\phi(y', y_{4}+c) $ is radially symmetric	and $|y|=|x| $, we have
		\begin{equation*}
			\frac{\psi(x_{1}+c, x_{2})}{x_{2}}=\varphi(|x|).
		\end{equation*}
		By translation of $ \psi $ for the $x_{1}$-variable,  we may assume that $ c=0 $. By the polar coordinate $ x_{1}=r\cos\theta$,  $x_{2}=r\sin\theta$,  we define
		\begin{equation*}
			\Psi(x)=\psi(x)-Wx_{2}=(\varphi(r)-W)r\sin\theta=:\eta(r)\sin\theta.
		\end{equation*}
		By \eqref{3.6},  $ \Psi $ satisfies
		\begin{equation}\label{3.9}
			\begin{cases}
				-\Delta \Psi+\Psi=\la \Psi_{+},  &\text{in} \, \, \, \Omega, \\
				-\Delta \Psi+\Psi=0,  &\text{in} \, \, \,  \Pi\backslash\Omega, \\
				\Psi=0,  & \text{on}\, \, \, \partial\Pi\cup\partial \Omega, \\
				\partial_{x_{1}}\Psi \rightarrow0,  \quad    \partial_{x_{2}}\Psi \rightarrow-W, \quad &\text{as} \, \, \, |x|\rightarrow\infty.
			\end{cases}
		\end{equation}
		where $ \Omega=B_{a}(0)\cap \Pi $ for some $ a>0 $.  Using $ \eqref{3.9}_{1} $ , we have
		\begin{equation}\label{3.10}
			\begin{cases}
				r^2\eta''+r\eta'+((\la-1)r^2-1)\eta-Wr^3=0, \ \eta>0, \  0<r<a, \\
				\eta(a)=0.\\
			\end{cases}
		\end{equation}
		We take $ \eta_{0}=\eta-\frac{W}{\la-1} r$,  then $ \eta_{0} $ satisfies
		\begin{equation}\label{3.11}
			\begin{cases}
				r^2\eta_{0}''+r\eta_{0}'+((\la-1)r^2-1)\eta_{0}=0, \ \eta_{0}(r)>-\frac{W}{\la-1}r, \  0<r<a, \\
				\eta_{0}(a)=-\frac{W}{\la-1}a.\\
			\end{cases}
		\end{equation}
		Since $ \eta_0(0) $ is bounded, we have 
		\begin{equation*}
			\eta _0=-\frac{Wa}{\lambda -1}\cdot \frac{J_1\Big( \big( \lambda -1 \big) ^{1/2}r \Big)}{J_1\Big( \big( \lambda -1 \big) ^{1/2}a \Big)}.
		\end{equation*}
		Similarly,  in	$ \Pi\backslash\Omega $,  $\eta  $ satisfies
		\begin{equation*}
			r^{2}\eta''+r\eta '-(r^{2}+1)\eta-Wr^{3}=0.
		\end{equation*}
		We take $ \eta_{1}=\eta+Wr $,  then $ \eta_{1} $ satisfies
		\begin{equation*}
			r^{2}\eta_{1}''+r\eta_{1}'-(r^{2}+1)\eta_{1}=0.
		\end{equation*}
		Since $ \eta_{1} $ is decaying at $ \infty $ and $ \eta(a)=0 $, we obtain
		\begin{equation*}
			\eta_{1}=\frac{Wa}{K_{1}(a)}K_{1}(r).
		\end{equation*}
		By $\Psi>0 $ in $B_{a}(0)\cap \Pi$ and the continuity of $ \partial_{r}\Psi $ at $a$,   it follows that	
		$a$ is the smallest positive solution of the following equation
		\begin{equation}\label{eq:3.12}
			a\Big(\frac{K_{1}'(a)}{K_{1}(a)}+\frac{1}{(\la-1)^{1/2}}\cdot\frac{J_{1}'((\la-1)^{1/2}a)}{J_{1}((\la-1)^{1/2}a)}\Big)=\frac{\la}{\la-1}.
		\end{equation}
		Hence we get
		\begin{equation*}
			\Psi (x)=\Psi _L(x)-Wx_2=\begin{cases}
				\Big( A_LJ_1\big((\lambda -1)^{1/2}r\big)+\frac{W}{\lambda -1}r \Big) \sin\theta , \quad& r\le a,\\
				\left( \frac{Wa}{K_1(a)}K_1(r)-Wr \right) \sin\theta, &		r>a,\\
			\end{cases}
		\end{equation*}	
		where 
		\begin{equation*}
			A_L=-\frac{Wa}{\lambda -1}\cdot \frac{1}{J_1((\lambda -1)^{1/2}a)}.
		\end{equation*}
	\end{proof}
	
	\begin{remark}\label{remark:3.5}
		We want to show that equation \eqref{eq:3.12} is solvable. Define the set as follows
		$$A=\Big\{ t\in \mathbb{R} _+ \mid  J_1\big( ( \lambda -1 ) ^{1/2}t  \big) \ne 0 \Big\} $$   and the function
		\begin{equation}
			\mathcal{W} \left( t \right) =\ln \frac{K_1\left( t \right) \cdot |J_1( \left( \lambda -1 \right) ^{1/2}t ) |^{1/\left( \lambda -1 \right)}}{t^{\lambda /\left( \lambda -1 \right)}},\quad  t \,\in \,A.
		\end{equation}
		By the properties of $ J_{1} $,  we know that $ \mathbb{R}_{+}\backslash A $ is at most countable. Suppose
		\begin{equation*}
			\mathbb{R}_{+}\backslash A=\left\{ x_1,x_2,\cdots ,x_{{n}},\cdots \right\}, \quad \mathrm{for} \,\, x_{{i+1}}>x_{{i}}>0, \quad  i\in \left\{ 1,2,3,\cdots \right\}.
		\end{equation*}
		We find that
		\begin{equation*}
			\lim_{{t\rightarrow x_{\mathrm{i}}}}\mathcal{W} \left( t \right) =-\infty,
		\end{equation*}
		and
		\begin{equation*}
			\mathcal{W} \left( t \right) >-\infty ,\quad   \mathrm{for} \,\, t\in\left( x_{\mathrm{i}}, x_{\mathrm{i}+1} \right),
		\end{equation*}
		where $ i\in \left\{ 1,2,3,\cdots \right\}$. Therefore, on each interval $ (x_{i}, x_{i+1}) $, $ \mathcal{W} $ has at least one extreme point, then \eqref{eq:3.12} is solvable. By direct calculation, we obtain
		\begin{equation*}
			\mathcal{W} '\left( t \right) =\frac{K_1'(t)}{K_1(t)}+\frac{1}{(\lambda -1)^{1/2}}\cdot \frac{J_1'((\lambda -1)^{1/2}t)}{J_1((\lambda -1)^{1/2}t)}-\frac{\lambda}{\lambda -1}\cdot \frac{1}{t},
		\end{equation*}
		and
		\begin{equation*}
			\lim_{t\rightarrow 0^+} \mathcal{W}'\left( t \right) =-\infty .
		\end{equation*}
		Thus there exists a smallest positive solution $a$ to equation \eqref{eq:3.12}.
	\end{remark}
	
	\begin{corollary}\label{add-cor}
		For $\lambda>1$, $\mu>0$ and $\nu\ge \nu_0$, we have
		\begin{equation*}
			\Sigma_{\mu, \nu, \lambda}=\left\{\omega_{L}^{\lambda, W}\left(\cdot+c \mathrm{e}_{1}\right) \mid c \in \mathbb{R} \right\},
		\end{equation*}
		where $W=\mu/I(\omega_L^{\lambda, 1})$.
	\end{corollary}

	\section{Compactness of Maximizing Sequences}\label{sect4}
	In this section, we shall prove the compactness of a  maximizing sequence up to translations for the $x_{1}$-variable by using a concentration compactness principle due to P. L. Lions.
	\begin{theorem}\label{lm4.1}
		Let $\lambda>1$, $\mu>0$ and $\nu\ge \nu_0$. Suppose that $ \{\omega _{n}\}_{n=1}^{\infty} $ is a maximizing  sequence in the sense that
		\begin{equation}\label{4.1}
			\omega _{n}\ge 0, \quad \omega _{n}\in L^{1}\cap L^{2}, \quad \int_{\Pi} \omega _{n} \mathrm{d} x \le \nu,  \quad \|\omega _{n}\|_{2}\le C, \quad \forall n\ge 1,
		\end{equation}
		\begin{equation}\label{4.2}
			\mu_{n}=\int_{\Pi}x_{2}\omega _{n}\mathrm{d} x \rightarrow \mu,  \quad    \text{as}~  n\rightarrow\infty,
		\end{equation}
		and
		\begin{equation}\label{	4.3}
			\mathcal{E}(\omega _{n})\rightarrow S_{\mu, \nu, \lambda}, \quad  \text{as}~n\rightarrow\infty.
		\end{equation}
		Then there exists $ \omega \in \Sigma_{\mu, \nu, \lambda} $,  a sub-sequence $ \{\omega _{n_{k}}\}_{k=1}^{\infty} $ and a sequence of real numbers $ \{c_{k}\}_{k=1}^{\infty} $ such that as $ k\rightarrow\infty $,  it holds
		\begin{equation}\label{4.4}
			\omega _{n_{k}}(\cdot+c_{k}\mathrm{e}_{1})\rightarrow \omega  \quad   \text{in} \, \, L^{2}(\Pi),
		\end{equation}
		and
		\begin{equation}\label{4.5}
			x_{2}\omega _{n_{k}}(\cdot+c_{k}\mathrm{e}_{1})\rightarrow x_{2}\omega  \quad  \text{in} \, \, L^{1}(\Pi).
		\end{equation}
	\end{theorem}
	
	To prove Theorem \ref{lm4.1}, we need the following concentration  compactness lemma (see \cite{lions1984concentration}).
	
	\begin{lemma}\label{lm4.2}
		Let $\left\{\xi_{n}\right\}_{n=1}^{\infty}$ be a sequence of nonnegative functions in $L^{1}(\Pi)$ satisfying
		$$
		\limsup _{n \rightarrow \infty} \int_{\Pi} \xi_{n} \mathrm{dx} \rightarrow \mu,
		$$
		for some $0<\mu<\infty$. Then, after passing to a subsequence, one of the following holds:
		\begin{enumerate}
			\item [(i)] (Compactness) There exists a sequence $\left\{y_{n}\right\}_{n=1}^{\infty}$ in $\overline{\Pi}$ such that for arbitrary $\varepsilon>0$, there exists $R>0$ satisfying
			$$
			\int_{\Pi \cap B_{R}\left(y_{n}\right)} \xi_{n} \mathrm{dx} \geq \mu-\varepsilon, \quad \forall n \geq 1.
			$$
			\item [(ii)] (Vanishing) For each $R>0$,
			$$
			\lim _{n \rightarrow \infty} \sup _{y \in \Pi} \int_{B_{R}(y) \cap \Pi} \xi_{n} \mathrm{~d} x=0 .
			$$
			\item [(iii)]  (Dichotomy) There exists a constant $0<\alpha<\mu$ such that for any $\varepsilon>0$, there exist $N=N(\varepsilon) \geq 1$ and $0 \leq \xi_{i, n} \leq \xi_{n}, i=1,2$ satisfying
			$$
			\left\{\begin{array}{l}
				\left\|\xi_{n}-\xi_{1, n}-\xi_{2, n}\right\|_{1}+\left|\alpha-\int_{\Pi} \xi_{1, n} \mathrm{dx}\right|+\left|\mu-\alpha-\int_{\Pi} \xi_{2, n} \mathrm{dx}\right|<\varepsilon_{,} \quad \text { for } n \geq N, \\
				d_{n}:=\operatorname{dist}\left(\operatorname{supp}\left(\xi_{1, n}\right), \operatorname{supp}\left(\xi_{2, n}\right)\right) \rightarrow \infty, \quad \text { as } \,  n \rightarrow \infty.
			\end{array}\right.
			$$
		\end{enumerate}

	\end{lemma}
	
	\begin{proof}[Proof of Theorem \ref{lm4.1}]

		Let $\xi_{n}=x_{2} \omega_{n}$. Using Lemma \ref{lm4.2},    we find that for a certain sub-sequence, still denoted by $\left\{\omega_{n}\right\}_{n=1}^{\infty}$, one of the three cases in Lemma \ref{lm4.2} should occur. To deal with the three cases, we divide the proof into three steps.
		
		Step 1. (Vanishing excluded) Suppose that for each fixed $R>0$,
		\begin{equation}\label{4.6}
			\lim _{n \rightarrow \infty} \sup _{y \in \Pi} \int_{B_{R}(y) \cap \Pi} x_{2} \omega_{n} \mathrm{~d} x=0.
		\end{equation}
		To get a contradiction, it is sufficient to prove that $\lim _{n \rightarrow \infty} E\left(\omega_{n}\right)=0$. We set
		\begin{equation*}
			\begin{split}
				2E(\omega _{n}) & =\int_{\Pi}\int_{\Pi}\omega _{n}(x){G}_{\Pi}(x, y)\omega _{n}(y)\mathrm{d} x\mathrm{d} y \\
				& =\Big(\iint_{|x-y|\ge R}+\iint_{|x-y|\le R}\Big)\omega _{n}(x){G}_{\Pi}(x, y)\omega _{n}(y)\mathrm{d} x\mathrm{d} y.
			\end{split}	
		\end{equation*}
		By \eqref{ess} we have
		\begin{equation*}
			\begin{split}
				\iint_{|x-y|\ge R}G_{\Pi}(x, y)\omega_{n} (x)\omega _{n}(y)\mathrm{d} x\mathrm{d}
				y & \le \iint_{|x-y|\ge R}G(x, y)\omega_{n} (x)\omega _{n}(y)\mathrm{d} x\mathrm{d}
				y \\
				& 	\le Ce^{-\frac{R}{2}}\nu ^{2}\to 0 \quad \text{as} \ R\rightarrow\infty.
			\end{split}
		\end{equation*}
		Set
		\begin{equation*}
			\iint_{|x-y|\le R}G_{\Pi}(x, y)\omega _{n}(x)\omega _{n}(y)\mathrm{d} x\mathrm{d} y= \Big( \iint_{ \substack  {|x-y|\le R\\ y_2\ge 1/R}}+\iint_{\substack  {|x-y|\le R\\ y_2< 1/R}} \Big)G_{\Pi}(x, y)\omega _{n}(x)\omega _{n}(y) \mathrm{d} x\mathrm{d} y .
		\end{equation*}
		By simple calculations, we get that
		\begin{equation*}
			\begin{split}
				\iint_{\substack  {|x-y|\le R\\ y_2\ge 1/R}}& G_{\Pi}(x, y)\omega _{n}(x)\omega _{n}(y)\mathrm{d} x\mathrm{d} y \le \int_\Pi \omega_n(x)\mathrm{d} x\int_{\substack  {|y-x|\le R\\ y_2\ge 1/R}}G(x, y)\omega_n(y)\mathrm{d} y \\
				& \le C\nu\|\omega_n \|_{2}^{1/2} \Big(\sup_{x\in \Pi}\int_{\substack  {|y-x|\le R\\ y_2\ge 1/R}}\omega_n(y)\mathrm{d} y \Big)^\frac{1}{2}\\
				&\le CR^\frac{1}{2} \Big(\sup_{x\in \Pi}\int_{B_R(x)}\omega_n(y)y_2\mathrm{d} y \Big)^\frac{1}{2}\to {0} \quad  \text{as}\ n\rightarrow\infty.
			\end{split}
		\end{equation*}
		In addition, we have
		\begin{equation*}
			\begin{split}
				\iint_{\substack  {|x-y|\le R\\ y_2< 1/R}}& G_{\Pi}(x, y)\omega _{n}(x)\omega _{n}(y)\mathrm{d} x\mathrm{d} y \le \int_\Pi \omega_n(x)\mathrm{d} x\int_{\substack  {|y-x|\le R\\ y_2<1/R}}G(x, y)\omega_n(y)\mathrm{d} y \\
				& \le C\nu\|\omega_n \|_{2}^{1/2} \Big(\sup_{x\in \Pi}\int_{\substack  {|y-x|\le R\\ y_2< 1/R}}G^2(x, y)\mathrm{d} y \Big)^\frac{1}{2}\\
				&\le C\nu \Big(\sup_{x\in \Pi}\int_{ y_2< 1/R}G^2(x, y)\mathrm{d} y \Big)^\frac{1}{2}\to 0\quad  \text{as} \ R\rightarrow\infty.
			\end{split}
		\end{equation*}
		Hence
		\begin{equation*}
			2E(\omega _{n})\le Ce^{-\frac{R}{2}}\nu ^{2}+CR^\frac{1}{2} \Big(\sup_{x\in \Pi}\int_{B_R(x)}\omega_n(y)y_2\mathrm{d} y \Big)^\frac{1}{2}+C\nu \Big(\sup_{x\in \Pi}\int_{ y_2< 1/R}G^2(x, y)\mathrm{d} y \Big)^\frac{1}{2}.
		\end{equation*}
		Letting $ n\rightarrow \infty $ and then $ R\rightarrow \infty $ implies $ \lim_{n\rightarrow\infty} E(\omega _{n})=0 $.

		Step 2. (Dichotomy excluded)
		Suppose that there exists some $ \alpha \in (0, \mu) $ such that
		\begin{equation}
			\left\{\begin{array}{l}
				\omega_{n}=\omega_{1, n}+\omega_{2, n}+\omega_{3, n} ,  \ 0\le \omega _{i, n}\le \omega  _{n}, \quad  i=1, 2, 3 , \\
				\left\|x_{2} \omega_{3, n}\right\|_{1}+\left|\alpha-\alpha_{n}\right|+\left|\mu-\alpha-\beta_{n}\right|\rightarrow0, \quad \text { as } n \rightarrow\infty , \\
				d_{n}:=\operatorname{dist}\left(\operatorname{supp}\left(\omega_{1, n}\right), \operatorname{supp}\left(\omega_{2, n}\right)\right) \rightarrow \infty, \quad \text { as } n \rightarrow \infty,
			\end{array}\right.
		\end{equation}
		where $\alpha_n=\|x_2\omega_{1, n}\|_1$ and  $\beta_n=\|x_2\omega_{2, n}\|_1$.
		According to the symmetry of $ E $,  we have
		$$
		\begin{aligned}
			2 E\left(\omega_{n}\right)&=2E\left(\omega_{1, n}+\omega_{2, n}+\omega_{3, n}\right) \\
			&=\int_{\Pi} \int_{\Pi} \omega_{1, n}(x) G_{\Pi}(x, y) \omega_{1, n}(y) \mathrm{d} x \mathrm{~d} y\\
			&\qquad+\int_{\Pi} \int_{\Pi} \omega_{2, n}(x) G_{\Pi}(x, y) \omega_{2, n}(y) \mathrm{d} x \mathrm{~d} y
			+2 \int_{\Pi} \int_{\Pi} \omega_{1, n}(x) G_{\Pi}(x, y) \omega_{2, n}(y) \mathrm{d} x \mathrm{~d} y\\
			&\qquad+\int_{\Pi} \int_{\Pi}\left(2 \omega_{n}-\omega_{3, n}(x)\right) G_{\Pi}(x, y) \omega_{3, n}(y) \mathrm{d} x \mathrm{~d} y .
		\end{aligned}
		$$
		For fixed $R>0$,
		$$
		\begin{aligned}
			\int_{\Pi} \int_{\Pi} &\left(2 \omega_{n}-\omega_{3, n}(x)\right) G_{\Pi}(x, y) \omega_{3, n}(y) \mathrm{d} x \mathrm{~d} y \\
			&\leq C\int_{y_2\ge 1/R} G(x, y) \omega_{3, n}(y)\mathrm{d} y+C\int_{y_2< 1/R} G(x, y) \omega_{3, n}(y)\mathrm{d} y\\
			&\le CR^\frac{1}{2}\left\|x_{2} \omega_{3, n}\right\|_{1}^{\frac{1}{2}}+C\Big(\sup_{x\in \Pi}\int_{ y_2< 1/R}G^2(x, y)\mathrm{d} y \Big)^\frac{1}{2}.
		\end{aligned}
		$$
		By \eqref{ess}, we have
		$$
		\int_{\Pi} \int_{\Pi} \omega_{1, n}(x) G_{\Pi}(x, y) \omega_{2, n}(y)\mathrm{d} x \mathrm{d} y \leq Ce^{-d_n/2}.
		$$
		Hence
		\begin{equation*}
			\begin{split}
				\mathcal{E}_{\la}(\omega _{n})&=E(\omega _{n})-\frac{1}{2\la}\int_{\Pi}\omega _{n}^{2}dx\\
				&\le \mathcal{E}_{\la}(\omega _{1, n})+\mathcal{E}_{\la}(\omega _{2, n})+CR^\frac{1}{2}\left\|x_{2} \omega_{3, n}\right\|_{1}^{\frac{1}{2}}+C\Big(\sup_{x\in \Pi}\int_{ y_2< 1/R}G^2(x, y)\mathrm{d} y \Big)^\frac{1}{2}+Ce^{-d_n/2}.
			\end{split}
		\end{equation*}
		Taking Steiner symmetrization $\omega_{i, n}^{*}$ of $\omega_{i, n}$ for $i=1,2$, we get
		$$
		\left\{ \begin{array}{l}
			\mathcal{E} _{\lambda}\left( \omega _n \right) \le \mathcal{E} _{\lambda}\left( \omega _{1,n}^{*} \right) +\mathcal{E} _{\lambda}\left( \omega _{2,n}^{*} \right) +CR^{\frac{1}{2}}\left\| x_2\omega _{3,n} \right\| _{1}^{\frac{1}{2}}+C\left( \mathop {\mathrm{sup}} \limits_{x\in \Pi}\int_{y_2<1/R}{G^2}(x,y)\mathrm{d}y \right) ^{\frac{1}{2}}+Ce^{-d_n/2},\\
			\left\| \omega _{1,n}^{*} \right\| _1+\left\| \omega _{2,n}^{*} \right\| _1\le \nu ,\quad \left\| \omega _{1,n}^{*} \right\| _2+\left\| \omega _{2,n}^{*} \right\| _2\le C,\\
			\left\| x_2\omega _{1,n}^{*} \right\| _1=\alpha _n,\quad \,\,\left\| x_2\omega _{2,n}^{*} \right\| _1=\beta _n.\\
		\end{array} \right. 
		$$
		We  assume that $\omega_{i, n}^{*} \rightarrow \omega_{i}^{*}$ weakly in $L^{2}(\Pi)$ as $ n\rightarrow \infty $ for $i=1,2$. Proceeding as in the proof  of Lemma \ref{lm2.5}, we can obtain	the convergence of the kinetic energy
		\begin{equation*}
			\lim _{n \rightarrow \infty} E\left(\omega_{i, n}^{*}\right)=E\left(\omega_{i}^{*}\right), \quad  \text { for } i=1,2.
		\end{equation*}
		By first letting $n \rightarrow \infty$, then $R \rightarrow \infty$, we obtain
		$$
		\left\{\begin{array}{l}
			S_{\mu, \nu, \lambda} \leq \mathcal{E}_{\la}\left(\omega_{1}^{*}\right)+\mathcal{E}_{\la}\left(\omega_{2}^{*}\right), \\
			\left\|\omega_{1}^{*}\right\|_{1}+\left\|\omega_{2}^{*}\right\|_{1} \leq \nu,\quad \left\|\omega_{1}^{*}\right\|_{2}+\left\|\omega_{2}^{*}\right\|_{2} \leq C, \\
			\left\|x_{2} \omega_{1}^{*}\right\|_{1} \leq \alpha,\quad \left\|x_{2} \omega_{2}^{*}\right\|_{1} \leq \mu-\alpha.
		\end{array}\right.
		$$
		We set $\alpha_{1}=\left\|x_{2} \omega_{1}^{*}\right\|_{1} \leq \alpha$, $\nu_{1}=\left\|\omega_{1}^{*}\right\|_{1}$, $\beta_{1}=\left\|x_{2} \omega_{2}^{*}\right\|_{1} \leq \mu-\alpha$ and $\nu_{2}=\left\|\omega_{2}^{*}\right\|_{1}$. It holds
		$$
		\alpha_{1}>0,\quad   \beta_{1}>0 .
		$$
		In fact, suppose that $\alpha_{1}=0$, then we have $\omega_{1}^{*} \equiv 0$, and hence
		$$
		S_{\mu, \nu, \lambda} \leq \mathcal{E}_{\la}\left(\omega_{1}^{*}\right)+\mathcal{E}_{\la}\left(\omega_{2}^{*}\right) \leq \mathcal{E}_{\la}\left(\omega_{2}^{*}\right) \leq S_{\beta_1, \nu, \lambda}.
		$$
		This is a contradiction to Lemma \ref{lm2.6}. Similarly, one can verify $\beta_{1}>0$. We choose $\hat{\omega}_{1} \in \Sigma_{\alpha_{1}, \nu_{1},\lambda}, \hat{\omega}_{2} \in \Sigma_{\beta_{1}, \nu_{2},\lambda}$. By Lemma \ref{lm3.3}, we have that supports of $\hat{\omega}_{i}, i=1,2$ are bounded. Therefore, we may assume that $\operatorname{supp}\left(\hat{\omega}_{1}\right) \cap \operatorname{supp}\left(\hat{\omega}_{2}\right)=\varnothing$ by suitable translations in $x_{1}$-direction. Letting $\hat{\omega}=\hat{\omega}_{1}+\hat{\omega}_{2}$, then we have
		\begin{equation*}
			\left\{\begin{array}{l}
				\int_{\Pi} \hat{\omega} \mathrm{d} x=\int_{\Pi} \hat{\omega}_{1} \mathrm{~d} x+\int_{\Pi} \hat{\omega}_{2} \mathrm{~d} x \leq \nu,  \\
				\int_{\Pi} x_{2} \hat{\omega} \mathrm{d} x=\int_{\Pi} x_{2} \hat{\omega}_{1} \mathrm{~d} x+\int_{\Pi} x_{2} \hat{\omega}_{2} \mathrm{~d} x=\alpha_{1}+\beta_{1} \leq \mu,
			\end{array}\right.
		\end{equation*}
		which implies that $\hat{\omega} \in \mathcal{A}_{\alpha_{1}+\beta_{1}, \nu}$. Observing that $\hat{\omega}_1\not \equiv 0$ and $\hat{\omega}_{2} \not \equiv 0$,  we have
		$$
		\begin{aligned}
			S_{\mu, \nu, \lambda} & \leq \mathcal{E}_{\la}\left(\omega_{1}^{*}\right)+\mathcal{E}_{\la}\left(\omega_{2}^{*}\right) \\
			&\leq \mathcal{E}_{\la}\left(\hat{\omega}_{1}\right)+\mathcal{E}_{\la}\left(\hat{\omega}_{2}\right)\\
			&=\mathcal{E}_{\la}(\hat{\omega})- 2\int_{\Pi} \int_{\Pi} \hat{\omega}_{1}(x) G_{\Pi}(x, y) \hat{\omega}_{2}(y) \mathrm{d} x \mathrm{d} y \\
			&<S_{\alpha_{1}+\beta_{1}, \nu, \lambda} \leq S_{\mu, \nu, \lambda},
		\end{aligned}
		$$
		which is a contradiction.

		Step 3. (Compactness) Assume that there is a sequence $\left\{y_{n}\right\}_{n=1}^{\infty}$ in $\overline{\Pi}$ such that for arbitrary $\varepsilon>0$, there exists $R>0$ satisfying
		\begin{equation}\label{4.8}
			\int_{\Pi \cap B_{R}\left(y_{n}\right)} x_{2} \omega_{n} \mathrm{~d} x \geq \mu-\varepsilon, \quad \forall\, n \geq 1.
		\end{equation}
		We may assume that $y_{n}=\left(0, y_{n, 2}\right)$ after a suitable $x_{1}$-translation. We claim that
		\begin{equation}\label{4.9}
			\sup _{n \geq 1} y_{n, 2}<\infty .
		\end{equation}
		Indeed,  if \eqref{4.9} is false,  then  there exists a subsequence, still denoted by $\left\{y_{n, 2}\right\}$, such that
		$$
		\lim _{n \rightarrow \infty} y_{n, 2}=\infty .
		$$
		By direct calculation, we have
		$$
		\begin{aligned}
			2 E\left(\omega_{n}\right) &=\int_{\Pi} \omega_{n}(x) \mathcal{G} \omega_{n}(x) \mathrm{d} x \\
			&=\int_{\Pi \cap B_{R}\left(y_{n}\right)} \omega_{n}(x) \mathcal{G} \omega_{n}(x) \mathrm{d} x+\int_{\Pi \backslash B_{R}\left(y_{n}\right)} \omega_{n}(x) \mathcal{G}\omega_{n}(x) \mathrm{d} x .
		\end{aligned}
		$$
		Since $\left\{\omega_{n}\right\}_{n=1}^{\infty}$ is uniformly bounded in $L^{2}(\Pi)$,  $\left\|x_{2} \omega_{n}\right\|_{1} \leq \mu+o(1)$ and \eqref{2.5},  we have
		$$
		\int_{\Pi \cap B_{R}\left(y_{n}\right)} \omega_{n}(x) \mathcal{G} \omega_{n}(x) \mathrm{d} x \leq \frac{C\mu}{(y_{n, 2}+1-R)^{1/2}}\to 0 \quad  \text{as} \  n\to\infty.
		$$
		For any fixed $M>0$ large, we have
		\begin{equation}\label{4.10}
			\begin{split}
				&\int_{\Pi \backslash B_{R}\left(y_{n}\right)} \omega_{n}(x) \mathcal{G}\omega_{n}(x) \mathrm{d} x \\
				&\leq C\int_{\substack  { \Pi \backslash B_{R}\left(y_{n}\right)\\ y_2< 1/M}} G(x, y) \omega_{ n}(y)\mathrm{d} y+C\int_{\substack  { \Pi \backslash B_{R}\left(y_{n}\right)\\ y_2< 1/M}} G(x, y) \omega_{3, n}(y)\mathrm{d} y\\
				&\le CM^\frac{1}{2}\left\|x_{2} \omega_{ n}1_{B_{R}\left(y_{n}\right) }\right\|_{1}^{\frac{1}{2}}+C\Big(\sup_{x\in \Pi}\int_{ y_2< 1/M}G^2(x, y)\mathrm{d} y \Big)^\frac{1}{2} \\
				&\le CM^\frac{1}{2}\varepsilon^\frac{1}{2}+C\Big(\sup_{x\in \Pi}\int_{ y_2< 1/M}G^2(x, y)\mathrm{d} y \Big)^\frac{1}{2}.
			\end{split}
		\end{equation}
		Hence, by first letting $n \rightarrow \infty$, then $\varepsilon \rightarrow 0$ and lastly $M \rightarrow \infty$, we obtain
		$$
		0<S_{\mu, \nu, \lambda}\leq \lim _{n \rightarrow \infty} E\left(\omega_{n}\right)=0.
		$$
		The claim \eqref{4.9} is thus proved. We may assume that  $ y_{n, 2}=0 $ by taking $ R $ larger. Therefore,  we have
		$$
		\int_{\Pi \cap B_{R}(0)} x_{2} \omega_{n} \mathrm{d} x \geq \mu-\varepsilon, \quad \forall\, n \geq 1 .
		$$
		Since $\left\{\omega_{n}\right\}$ is uniformly bounded in $L^{2}$, by choosing a subsequence, $\omega_{n} \rightarrow \omega$ weekly in $L^{2}$ for some $\omega$. By sending $n \rightarrow \infty$,
		$$
		\int_{\Pi}\omega \mathrm{d} x\le \nu, \quad 
		\int_{\Pi} x_{2} \omega \mathrm{d} x=\mu.
		$$
		Hence $\omega \in \mathcal{A}_{\mu, \nu}$. Let us assume that
		\begin{equation}\label{4.11}
			\lim _{n \rightarrow \infty} E\left(\omega_{n}\right)=E(\omega),
		\end{equation}
		which implies
		\begin{align*}
			S_{\mu, \nu, \lambda}&=\lim _{n \rightarrow \infty} \mathcal{E}_{\la}\left(\omega_{n}\right)\\
			&\leq \lim _{n \rightarrow \infty} E\left(\omega_{n}\right)-\frac{1}{2\la} \liminf _{n \rightarrow \infty}\left\|\omega_{n}\right\|_{2}^{2} \\
			&\leq \mathcal{E}_{\la}(\omega)\leq S_{\mu, \nu, \lambda}.
		\end{align*}
		Hence $\lim _{n \rightarrow \infty}\left\|\omega_{n}\right\|_{2}=\|\omega\|_{2}$ and $\omega_{n} \rightarrow \omega$ in $L^{2}$ follows. By
		\begin{align*}
			\int_{\Pi} x_{2}\left|\omega_{n}-\omega\right| \mathrm{d} x &=\int_{\Pi \cap B_{R}(0)} x_{2}\left|\omega_{n}-\omega\right| \mathrm{d} x+\int_{\Pi \backslash B_{R}(0)} x_{2}\left|\omega_{n}-\omega\right| \mathrm{d} x \\
			& \leq C R^{2}\left\|\omega_{n}-\omega\right\|_{2}+\int_{\Pi \backslash B_{R}(0)} x_{2}\left(\omega_{n}+\omega\right) \mathrm{d} x \\
			& \leq C R^{2}\left\|\omega_{n}-\omega\right\|_{2}+\mu_{n}-\mu+2 \varepsilon .
		\end{align*}
		Sending $n \rightarrow \infty$ and then $\varepsilon \rightarrow 0$,  the above inequality implies $x_{2} \omega_{n} \rightarrow x_{2} \omega$ in $L^{1}(\Pi)$. Since $\mathcal{E}_{\la}\left(\omega_{n}\right)\rightarrow \mathcal{E}_{\la}(\omega)$, the limit $\omega \in \mathcal{A}_{\mu, \nu}$ is a maximizer of $S_{\mu, \nu}$.
		
		It remains to show the assumption \eqref{4.11}. On the one hand,  for any fixed $ M>0 $ large,  we have
		\begin{equation*}
			\begin{split}
				2E(\omega _{n})&=\int_{\Pi}\int_{\Pi}\omega _{n}(x)G_{\Pi}(x, y)\omega _{n}(y)\mathrm{d}x\mathrm{d}y\\
				&\le \int_{\Pi\cap B_{R}(0)}\int_{\Pi\cap B_{R}(0)}\omega _{n}(x)G_{\Pi}(x, y)\omega _{n}(y)\mathrm{d}x\mathrm{d}y\\
				&\ \ \ \ \ \ \ +2\int_{\Pi\backslash B_{R}(0)}\int_{\Pi}\omega _{n}(x)G_{\Pi}(x, y)\omega _{n}(y)\mathrm{d}x\mathrm{d}y\\
				&\le \int_{\Pi\cap B_{R}(0)}\int_{\Pi\cap B_{R}(0)}\omega _{n}(x)G_{\Pi}(x, y)\omega _{n}(y)\mathrm{d}x\mathrm{d}y\\
				& \ \ \ \ \ \ \  \ \ +CM^\frac{1}{2}\left\|x_{2} \omega_{ n}1_{\Pi\backslash B_{R}(0)}\right\|_{1}^{\frac{1}{2}}+C\left(\sup_{x\in \Pi}\int_{ y_2< 1/M}G^2(x, y)\mathrm{d} y \right)^\frac{1}{2} .
			\end{split}
		\end{equation*}
		Letting $ n\rightarrow\infty $,  then $ \ep\rightarrow0 $ and lastly $ M\rightarrow\infty $,  we get
		$$
		\limsup _{n \rightarrow \infty} E\left(\omega_{n}\right) \leq E(\omega).
		$$
		On the other hand, for any $L>0$, we have
		\begin{align*}
			2 E\left(\omega_{n}\right)&=\int_{\Pi} \int_{\Pi} \omega_{n}(x) G_{\Pi}(x, y) \omega_{n}(y) \mathrm{d} x \mathrm{d} y \\
			&\geq \int_{\Pi \cap B_{L}(0)} \int_{\Pi \cap B_{L}(0)} \omega_{n}(x) G_{\Pi}(x, y) \omega_{n}(y) \mathrm{d} x \mathrm{d} y,
		\end{align*}
		which implies
		\begin{equation*}
			\liminf_{n \rightarrow \infty}E(\omega _{n})\ge E(\omega ).
		\end{equation*}
		The proof of \eqref{4.11} is thus completed.
	\end{proof}
	
	\section{Orbital Stability}\label{sect5}
	In this section, we establish the orbital stability of the Lamb dipoles $\omega_{L}$. Recalling Corollary \ref{add-cor}, Theorem \ref{thm1.4} follows from the following result.
	\begin{theorem}\label{lm5.1}
		Let $\lambda>1$, $\mu>0$ and $\nu\ge \nu_0$. Then for any $\varepsilon>0$, there exists $\delta>0$ such that for any non-negative function $\zeta_{0} \in L^{1} \cap L^{2}(\Pi)$ and
		\begin{equation*}
			\inf _{\omega \in \Sigma_{\mu, \nu, \lambda}}\left\{\left\|\zeta_{0}-\omega\right\|_{L^1\cap L^2}+\left\|x_{2}\left(\zeta_{0}-\omega\right)\right\|_{L^1}\right\} \leq \delta,
		\end{equation*}
		if there exists a $L^{2}$-regular solution $\zeta(t)$ with initial data ${\zeta_{0}}$, then
		\begin{equation}\label{5.2}
			\inf _{\omega \in \Sigma_{\mu, \nu, \lambda}}\left\{\left\|\zeta(t)-\omega\right\|_{L^1\cap L^2}+\left\|x_{2}\left(\zeta(t)-\omega\right)\right\|_{L^1}\right\} \leq \delta
		\end{equation}
		for all $t\ge 0$.
	\end{theorem}
	\begin{proof}
		We argue by contradiction. Suppose that the statement were false. Then there exists $\varepsilon_{0}>0$ such that for $n \geq 1$, there exist $\zeta_{0, n} \in L^{1} \cap L^{2}(\Pi)$ satisfying
		\begin{equation*}
			\inf _{\omega \in \Sigma_{\mu, \nu, \lambda}}\left\{\left\|\zeta_{0, n}-\omega\right\|_{L^1\cap L^2}+\left\|x_{2}\left(\zeta_{0, n}-\omega\right)\right\|_{L^1}\right\} \leq \frac{1}{n},
		\end{equation*}
		and
		\begin{equation}\label{5.3}
			\inf _{\omega \in \Sigma_{\mu, \nu, \lambda}}\left\{\left\|\zeta(t)-\omega\right\|_{L^1\cap L^2}+\left\|x_{2}\left(\zeta(t)-\omega\right)\right\|_{L^1}\right\} \ge \varepsilon_0,
		\end{equation}
		where $ \zeta_{n}(t) $ is a $ L^{2} $-regular solution with the initial data $ \zeta_{0, n} $.
		We take $\omega_{n} \in \Sigma_{\mu, \nu, \lambda}$ such that
		\begin{equation*}
			\left\|\zeta_{0, n}-\omega_n\right\|_{L^1\cap L^2}+\left\|x_{2}\left(\zeta_{0, n}-\omega_n\right)\right\|_{L^1}\to 0	\quad  \text{as}\ n\rightarrow\infty.
		\end{equation*}
		It is not hard to verify that
		$$
		\mathcal{E}_{\lambda}\left(\zeta_{0, n}\right) \rightarrow S_{\mu, \nu, \lambda}.
		$$
		We write $\zeta_{n}=\zeta_{n}\left(t_{n}\right)$ by suppressing $t_{n}$.    By the conservation laws,  one has
		\begin{equation*}
			\left\{ \begin{array}{l}
				\zeta _n\ge 0,\,\,\zeta _n\in L^1\cap L^2(\Pi ),\,\,\int_{\Pi}{\zeta _n}\mathrm{d}x\le \nu ,\,\,\left\| \zeta _n \right\| _2\le C,\\
				\mu _n=\int_{\Pi}{x_2}\zeta _n\mathrm{d}x\rightarrow \mu , \, \,\,\,  \, \mathrm{as} \,\, n\rightarrow \infty ,\\
				\mathcal{E}_{\lambda}\left( \zeta _n \right) \rightarrow S_{\mu ,\nu, \lambda}, \qquad  \quad   \mathrm{as} \,\, n\rightarrow \infty .\\
			\end{array} \right. 
		\end{equation*}
		By Theorem \ref{lm4.1},  there exist $ \omega \in \Sigma_{\mu, \nu,\ \lambda} $,  a subsequence $ \{\zeta_{n_{k}}\} _{k=1}^{\infty}$ and a sequence of real number $ \{c_{k}\}_{k=1}^{\infty} $ such that
		$$
		\left\|\zeta_{n_{k}}\left(\cdot+c_{k} \mathbf{e}_{1}\right)-\omega\right\|_{2}+\left\|x_{2}\left(\zeta_{n_{k}}\left(\cdot+c_{k} \mathbf{e}_{1}\right)-\omega\right)\right\|_{1} \to 0, \quad  \text{as}\ k\rightarrow\infty,
		$$
		which is contrary to \eqref{5.3},  and the proof of Theorem \ref{lm5.1} is thus completed.
	\end{proof}

	\subsection*{Acknowledgment}
	
	\noindent  { This work was supported by NNSF of China Grant 11831009 and 12201525. }

\end{document}